\documentclass{amsart}

\usepackage{amsmath}
\usepackage{amssymb}
\usepackage{amsthm}
\usepackage{xcolor}
\usepackage{orcidlink}
\usepackage{amssymb}
\usepackage{amsthm}
\usepackage{xcolor}
\usepackage{orcidlink}
\usepackage{comment}
\usepackage{hyperref}
\numberwithin{equation}{section}
\newtheorem{thm}{Theorem}[section]
\newtheorem{coro}[thm]{Corollary}
\newtheorem{prop}[thm]{Proposition}
\newtheorem{lem}[thm]{Lemma}

\newtheorem{rem}[thm]{Remark}

\DeclareMathOperator{\id}{I}

\newcommand{\D}{\mathbb{D}}

\newcommand{\N}{\mathbb{N}}
\newcommand{\R}{\mathbb{R}}
\newcommand{\T}{\mathbb{T}}
\newcommand{\Hpd}{H^p\left(\mathbb{D}^d\right)}
\newcommand{\vertiii}[1]{{\left\vert\kern-0.25ex\left\vert\kern-0.25ex\left\vert #1 
    \right\vert\kern-0.25ex\right\vert\kern-0.25ex\right\vert}}

\newcommand{\quilt}{\mathcal{R}}

\title{Simply interpolating and Carleson sequences for Hardy spaces in the polydisc}

\author{Nikolaos Chalmoukis
\orcidlink{0000-0001-5210-8206}
}
\address{Dipartimento di Matematica e Applicazioni, Universit\'a degli studi di Milano Bicocca, via Roberto Cozzi, 55 20125, Milano, Italy}
\email{nikolaos.chalmoukis@unimib.it}
\thanks{N. Chalmoukis is a  member of the Gruppo Nazionale per l'Analisi Matematica, la Probabilit\`a e le loro Applicazioni (GNAMPA) of the
Istituto Nazionale di Alta Matematica (INdAM)}

\author{Alberto Dayan \orcidlink{0000-0002-7346-4354}
}
\address{Fachrichtung Mathematik Universit\"at des Saarlandes, 66123 Saarbr\"ucken, Germany
}\email{dayan@math.uni-sb.de}
\date{}
\thanks{A. Dayan is partially supported by the Emmy Noether Program of the
German Research Foundation (DFG Grant 466012782)}


\begin{document}

\begin{abstract}
    We study the relation between simply and universally interpolating sequences for the holomorphic Hardy spaces $H^p(\mathbb{D}^d)$ on the polydisc. In dimension $d=1$ a sequence is simply interpolating if and only if it is universally interpolating, due to a classical theorem of Shapiro and Shields. In dimension $d\ge2$, Amar showed that Shapiro and Shields' theorem holds for $H^p(\mathbb{D}^d)$ when $p \geq 4$. In contrast, we  show that if $1\leq p \leq 2$ there exist simply interpolating sequences which are not universally interpolating. 
\end{abstract}
\maketitle
\section{Introduction}
Let $ f $ be a holomorphic function in the $d$-dimensional polydisc $\mathbb{D}^d$.
We say that $f$ belongs to the Hardy space $H^p(\mathbb{D}^d)$, if  
\[ \Vert f\Vert_p^p : = \sup_{0<r<1} \int_{\mathbb{T}^d}|f(rz)|^pdm(z)<+\infty,   \]
where $\mathbb{T}^d$ is the $d$-dimensional torus and $m$ is its probability Haar measure. We will also occasionally use the notation $|\,\,\, |$ to denote $m.$
The spaces of bounded holomorphic functions $H^\infty(\mathbb{D}^d)$ is defined analogously.

The present paper grew out of the attempt to better understand interpolating sequences for Hardy spaces in the higher dimensional polydisc. 
In dimension one, the work of Carleson \cite{Carleson58} and Shapiro and Shields \cite{SS61} provides a solid understanding of interpolating sequences. 
Let us spend a few words to recall the basic definitions and results in this simpler case. 

Consider $\Lambda = (\lambda_n)_n$ a sequence of points in the unit disc $\D$. Then
\[
(1-|\lambda_n|^2)^{\frac1p}|f(\lambda_n)|\le \|f\|_p, \,\,\, f\in H^p(\D),
\]
hence for $1\leq p < \infty$ the restriction operator at $\Lambda$ on $H^p(\D)$ is naturally defined as 

\[ T^p_\Lambda(f) = \big( f(\lambda_n)(1-|\lambda_n|^2)^{\frac1p} \big), \,\,\, f\in H^p(\mathbb{D}).  \]
A sequence is called {\it simply interpolating} if $T_\Lambda^p(H^p(\mathbb{D})) \supseteq \ell^p$. Explicitly, if for every sequence of {\it data} $ ( a_n ) \in \ell^p $ there exists an {\it interpolating function} $f\in H^p(\mathbb{D})$, 
in the sense that 
\begin{equation}\label{eq_interpolation}
	f(\lambda_n) = a_n(1-|\lambda_n|^2)^{-\frac1p}, \,\, n = 1,2,\dots
\end{equation}

It is worth mentioning that, since all Hardy spaces contain the polynomials, an interpolation problem like \eqref{eq_interpolation} always has a solution if the number of nodes is finite.  If $\Lambda$ is a simply interpolating, an application of the open mapping theorem yields that there exists a constant $C$, depending only on $\Lambda$ and $p$, such that one can find a solution $f$ for \eqref{eq_interpolation} such that $\Vert f \Vert_p \leq C \Vert (a_n) \Vert_{\ell^p}$ (see Lemma \ref{lemma:c_int}). 

A sequence is called {\it universally interpolating} if it is simply interpolating and $T_\Lambda^p$ maps boundedly $H^p(\mathbb{D})$ into $\ell^p.$ 
In other words a sequence is universally interpolating if for every $\ell^p$ data 
the interpolation problem \eqref{eq_interpolation} has a solution and there exists 
$M>0$ such that 
\begin{equation} \label{eq_Carleson} \sum_{n=1}^{\infty} |f(\lambda_n)|^p(1-|\lambda_n|^2) \leq M \Vert f\Vert^p_p, \,\, \forall f\in H^p(\mathbb{D}).  
\end{equation}

Sequences that satisfy \eqref{eq_Carleson}, are usually called {\it Carleson sequences}. 
If $\mu$ is a positive Borel measure in $\mathbb{D}$  we say that it is an $H^p(\mathbb{D})$-{\it Carleson measure} if there exists $M>0$ such that 
\begin{equation}
	\int_\mathbb{D}|f|^p d\mu \leq M \Vert f\Vert^p_p, \,\,\, \forall f \in H^p(\mathbb{D}).
\end{equation}
Hence $\Lambda$ is a Carleson sequence if and only if the atomic measure $\mu_\Lambda:=\sum_n(1-|\lambda_n|^2)~\delta_{\lambda_n}$
supported on $\Lambda$ is a Carleson measure.
It is a classical theorem of Carleson that the class of the homonymous measures is the same for every $p \geq 1 $ and a measure is such if and only if there exists $C>0$ such that 
\begin{equation}\label{eq:one_box} \mu(S(I)) \leq C |I|, \end{equation}
for every arc $I\subseteq \mathbb{T}$ and every region $S(I)$ of the form 
\[ S(I) = \{ z\in \mathbb{D}\setminus\{0\} : 1-|z| \leq |I|, z/|z| \in I \}. \]
Furthermore, Shapiro and Shields showed that a sequence is universally interpolating for $H^p(\mathbb{D})$ if it is {\it hyperbolically separated}, that is 
\begin{equation*}
	\inf_{n \neq m} \Big| \frac{\lambda_n-\lambda_m}{1-\overline{\lambda_m}\lambda_n} \Big|>0
\end{equation*}
and it is a Carleson sequence. Even more surprisingly, they proved that simply interpolating sequences are the same for all $p\ge1$ and that a simply interpolating sequence for $H^p(\mathbb{D})$  is automatically universally interpolating. In other words, if for some sequence $\Lambda \subseteq \mathbb{D}$ we have $T_\Lambda^p(H^p(\mathbb{D})) \supseteq \ell^p $ then $T_\Lambda^p(H^p(\mathbb{D})) = \ell^p. $     

In dimension $d\geq 2$ the analogous problems concerning $H^p(\mathbb{D}^d)$ functions can be stated using the following restriction operator. Let $\Lambda = (\lambda_n) \subseteq \mathbb{D}^d, \,\, \lambda_n=(\lambda_n^1,\dots,\lambda_n^d) $. Then define  
\begin{equation}
\label{eqn_restriction_operator}
T_\Lambda^p(f) = \Big( f(\lambda_n) \prod_{i=1}^d (1-|\lambda_n^i|^2)^{\frac1p} \Big), \qquad f \in H^p(\mathbb{D}^d). 
\end{equation}

If $p=\infty$, $T^\infty_\Lambda$ is the unweighted restriction operator on $H^\infty(\D^d)$. We say that a sequence $\Lambda$ is;
\begin{itemize}
\item[(SI)] \emph{simply interpolating} for $H^p(\mathbb{D}^d)$  if $T_\Lambda^p(H^p(\mathbb{D}^d))  \supseteq \ell^p $,
\item[(UI)] \emph{universally interpolating} for $H^p(\mathbb{D}^d)$ if $ T_\Lambda^p(H^p(\mathbb{D}^d))  = \ell^p$,  
\item[(CS)] a \emph{Carleson sequence} for $H^p(\mathbb{D}^d)$ if $T_\Lambda^p(H^p(\mathbb{D}^d)) \subseteq \ell^p.$ 
\end{itemize}

Similarly, a positive Borel measure on $\mathbb{D}^d$ is called Carleson if $H^p(\mathbb{D}^d)\subseteq L^p(\mathbb{D}^d,\mu).$ Then $\Lambda$ is a Carleson sequence if and only if the atomic measure 
\begin{equation}
    \label{eqn:muLambda}
    \mu_\Lambda:=\sum_n \left(\prod_{i=1}^d1-|\lambda_n^i|^2\right)~\delta_{\lambda_n}
\end{equation}
is a Carleson measure. As in the one dimensional case if $\Lambda$ is a simply interpolating sequence there exists $C>0$ such that for every $a\in \ell^p$, there exists $f\in H^p(\mathbb{D}^d)$ such that $T_\Lambda^p(f) = a $ and $\Vert f \Vert_p \leq C \Vert a \Vert_{\ell^p}$. The infimum of such $C$ is  the \emph{simple interpolation constant} of $\Lambda$. Similarly, given a Carleson sequence $\Lambda$ we define its \emph{Carleson constant} as the norm of $T^p_\Lambda: H^p(\mathbb{D}^d) \to \ell^p$.
The corresponding separation condition, also called { \it weak separation } becomes 
\[ \inf_{n\neq m } \max_{1\leq i\leq d}  \Big| \frac{\lambda^i_n-\lambda^i_m}{1-\overline{\lambda^i_n}\lambda^i_m} \Big|>0. \]

A first hint for how two or more dimensions are different than one came again from Carleson \cite{Carleson74}. He considered the  Hardy space $h^p(\mathbb{D}^d)$ of {\it separately harmonic} functions, which consists of functions $u\in C^\infty(\mathbb{D}^d) $  such that 
\[ \frac{\partial^2 u}{\partial \overline{z_i}\partial z_i} = 0, 1\leq i \leq d  \,\,\, \text{and} \,\,\,\, \Vert u \Vert^p_p = \sup_{0<r<1}\int_{\mathbb{T}^d}|u(rz)|^p dm(z) < + \infty. \]

In $d=1$, it is an immediate consequence of the M. Riesz Theorem that if $p>1$, then 
$H^p(\mathbb{D}) \subseteq L^p(\mathbb{D},\mu)$ if and only if $h^p(\mathbb{D}) \subseteq L^p(\mathbb{D},\mu)$.  
Consequently, a quite natural conjecture would be that $ h^p(\mathbb{D}^d) \subseteq L^p(\mathbb{D}^d,\mu) $ if and only if the $d$-dimensional version of \eqref{eq:one_box}, namely
\[ \mu(S(I_1)\times \dots S(I_d)) \leq C |I_1| \cdots |I_d|, \]
holds for all arcs $I_1,\dots I_d \subseteq \mathbb{D}$ and some fixed constant $C>0.$ 
Often this condition is referred to as the {\it one box} condition. 
It turns out that this is not the case. In fact, there exists a measure $\mu$ satisfying the one box condition but $ h^p(\mathbb{D}^2) \not \subseteq L^p(\mathbb{D}^2,\mu) $ for any $p>1$. M

As an intermediate step towards our results on interpolation, via a quite elementary 
modification of Carleson's example, we can show a little bit more; there exists a measure $\mu$ in $\mathbb{D}^2$  satisfying the one box condition but $H^p(\mathbb{D}^2) \not \subseteq L^p(\mathbb{D}^2,\mu)$ for any $p\geq 1$. 
The details are given in Section \ref{sec:Carlesons_Quilt}.  

Nonetheless, a complete characterization of Carleson measures for the holomorphic Hardy spaces in higher dimensions has yet to emerge. 
In particular it is not clear whether the sufficient condition of Chang - Stein \cite{Chang1979} is also necessary, or if the class of $H^p(\mathbb{D}^d)$ Carleson measures is the same for every $p\geq 1$ for that matter. We remark that often in the literature the necessity of the Chang - Stein condition has been assumed to be true, but to the best of our knowledge a proof is not known. The Chang-Stein condition characterizes Carleson measures for all Hardy spaces of separately harmonic functions on the polydisc. By inclusions, any Carleson measure for a separately harmonic space is Carleson for its holomorphic counterpart, but the converse implication is not known in dimension $d\ge2$. More details are provided in Section \ref{sec:Carleson}, where we show the equivalence of the two notion of Carlerson measures for a class of reflection-invariant measures.

Regarding interpolating sequences when $d\geq 2$ , the state of affairs is similarly unclear. Varopoulos showed in \cite{Varopoulos72} that an interpolating sequence for $H^\infty(\D^d)$ generates a measure that satisfies the Chang-Stein condition via \eqref{eqn:muLambda}. Nonetheless, Berndtsson et al. \cite{Berndtsson1987} gave a counterexample for the converse implication and Amar \cite{Amar80} showed  that a simply interpolating sequence for $H^2(\D^d)$ needs not to be inteprolating for $H^\infty(\D^d)$, in contrast with the one dimensional case. Finally, in \cite{Amar07} the author studies the relation between interpolating sequences for different Hardy spaces.
On the other hand, a much greater deal is known about the interpolating properties of sequences which are realizations of a certain random point process in the polydisc (see \cite{Dayan2023}, \cite{Chalmoukis2024} and \cite{Chalmoukis2022}).

	The question that motivated this work has been to which extend the Shapiro and Shields' theorem extends to the polydisc. That is, is it true that a simply interpolating sequence is also Carleson, and hence automatically universally interpolating? This is known to be false for the Dirichlet space in the unit disc \cite{Marshall1994}, 
	and for a large class of spaces in the unit disc and the unit ball in \cite{Chalmoukis24}. In this context we prove the following theorem.
\begin{thm}
    \label{thm:main}
    Let $d\ge2$.
     There exists a sequence $\Lambda_0 \subseteq \D^2$ that is simply interpolating for $H^p(\D^2)$ for all $1\le p\le2$, yet it is not a Carleson sequence for $H^r(\D^2)$ for any $r\ge1$.
\end{thm}
On the other hand, Amar showed  that if $p>2$ and $\Lambda$ is simply interpolating for $H^p(\D^d)$, then there exists a value of $r$ for which $\Lambda$ is Carleson for $H^r(\D^d)$.

\begin{thm}[\cite{Amar2020}]
\label{thm:Amar}
     Let $p>2$. If $\Lambda$ is simply interpolating for $H^p(\D^d)$, then $\Lambda$ is a Carleson sequence for $H^r(\D^d)$, where $1/p+1/r=1/2$.
\end{thm}
If $p\geq 4$ and $1/p+1/r=1/2$, then $r\le4$. Moreover, in Lemma \ref{lemma:interpolation_carleson} we will observe that if $\Lambda$ is a Carleson sequence for $H^p(\D^d)$ and $q>p$, then $\Lambda$ is a Carleson sequence also for $H^q(\D^d)$. Therefore from Theorem \ref{thm:main} and Theorem \ref{thm:Amar} we deduce the following.
\begin{coro}
Let $d\ge2$. If $p\geq 4$ and $\Lambda \subseteq \mathbb{D}^d$ is a simply interpolating sequence for $H^p(\mathbb{D}^d)$, then $\Lambda$ is also universally interpolating for $H^p(\mathbb{D}^d)$.

On the other hand, for all $1\leq p\le2$ there exists a sequence that is simply interpolating for $\Hpd$  but that is not universally interpolating for $H^p(\D^d)$.  
\end{coro}
It is evident that the theorem leaves the gap $2<p<4$ for which we can not answer the question on universal interpolating sequences completely. This deficiency stems once again from our poor understanding of Carleson measures for the polydisc.

The construction of the sequence $\Lambda_0$ is intimately connected to the counterexample of Carleson for measures which satisfy the one box condition but are not Carleson for $h^p(\mathbb{D}^2).$ The details of the construction are provided in Section \ref{sec:Carlesons_Quilt}. In Section \ref{sec:Carleson} we discuss the relation between Carleson measures for separately harmonic and holomorphic Hardy spaces, while Section \ref{sec:FangXia} contains the necessary tools to show that $\Lambda_0$ is simply interpolating for $1\le p\le2$. This will lead to the proof of Theorem \ref{thm:main}, which is discussed in Section \ref{sec:proof}. A simplified version of Amar's argument for the proof of Theorem \ref{thm:Amar} is contained in Section \ref{sec:amar}.

\section{Proof of Theorem \ref{thm:Amar}}
\label{sec:amar}
In what follows a recurring object is the Szeg\"o kernel in the polydisc.
That is, the following function; 
\begin{equation}
\label{eqn:szego}
S(z,w) = \prod_{i=1}^d\frac{1}{1-\overline{w^i}z^i} \qquad z=(z^1, \dots,z^d), w=(w^1, \dots, w^d) \in \mathbb{D}^d. 
\end{equation}
When considering $S$ as a (holomorphic) function of $z$ for fixed $w$ we will write $S_w$ instead. 
In fact $S$ is exactly the reproducing kernel of the reproducing kernel Hilbert space
$H^2(\mathbb{D}^d)$ with the standard inner product which we denote by $\langle \,\,\, , \,\, \rangle $. In particular 
\[ \Vert S_w \Vert^2_2 =\prod_{i=1}^d \frac{1}{1-|w^i|^2}, \,\,\, w\in \mathbb{D}^d. \]
We will also denote by $g_w$ the normalized Szeg\"o kernel  at $w$, that is 
\begin{equation}
\label{eqn:normalized_Szego}
g_w = S_w/\Vert S_w\Vert_2.
\end{equation}
The argument in \cite[Theorem 2.2]{Amar2020} makes implicitly use of the existence of a constant of interpolation, whose existence follows from the following standard lemma in functional analysis. Notice that since pointwise evaluations are bounded linear functionals on Hardy spaces the restriction operator $T^p_\Lambda$ is always closed. 
\begin{lem}
    \label{lemma:c_int}
    Let $T:X\to Y$ a closed and surjective linear operator between two Banach spaces $X,Y$. Then there exists $C>0$ such that for every $y\in Y$ there exists $x \in X$ such that \[ \Vert x \Vert_X \leq C\Vert y \Vert_Y, \,\,\, \text{and} \,\,\, Tx=y. \]
\end{lem}
\begin{proof} Let $D_T$ the linear submanifold of $X$ where the operator $T$ is defined. By the fact that $T$ is closed $T:D_T \to Y$ becomes a bounded surjective operator if we equip $D_T$ with the graph norm. Then the lemma follows by an application of the open mapping theorem. 
\end{proof}

   The proof of \cite[Theorem 2.2]{Amar2020} can be summarized as follows. Fix $N$ in $\mathbb{N}$, and consider the normalized kernel functions $(g_n)_{n=1}^N$ in $H^2(\mathbb{D}^d)$ associated to the points $(\lambda_n)^N_{n=1}$. Let $(h_n)_{n=1}^N$ be the minimal dual system of $(g_n)_{n=1}^N$ in $H^2(\mathbb{D}^d)$. Namely, each $h_n$ belongs to the linear span of $(g_n)_{n=1}^N$, and $\langle g_n, h_m\rangle=\delta_{n, m}$. Since the normalized kernel are linearly independent, then so is the collection of their dual system. Moreover, the projection $P_N$ of $H^2(\mathbb{D}^d)$ onto the linear span of  $(g_n)_{n=1}^N$ can be written as
\[
P_N(f):=\sum_{n=1}^N\langle f, g_n\rangle~h_n=\sum_{n=1}^N f(\lambda_n) \Vert S_{\lambda_n}\Vert_2 ^{-1} h_n\qquad f\in H^2(\mathbb{D}^d).
\]

Fix some $g\in H^r(\D^d)$. Since $\|h_n\|_{2}\ge1$ for all $n$, by Orlicz's Lemma (see \cite[Theorem 3.1.5]{Nikolski02}) one can find $(\varepsilon_n)_{n=1}^N\subseteq \mathbb{T}$ such that 
\[
\sum_{n=1}^N|g(\lambda_n)|^r \Vert S_{\lambda_n} \Vert^{-2}_2 \le \sum_{n=1}^N |g(\lambda_n)|^r \Vert S_{\lambda_n} \Vert^{-2}_2 \Vert h_n \Vert_2 ^2  \leq\left\|\sum_{n=1}^N\varepsilon_n |g(\lambda_n)|^{\frac{r}{2}} \Vert S_{\lambda_n} \Vert^{-1}_2 h_n\right\|_{2}^2.
\]

Since $\Lambda$ is a simply interpolating sequence for $H^p(\D^d)$, by Lemma \ref{lemma:c_int} we can find a function $f\in H^p(\mathbb{D}^d)$ such that 
\[ f(z_n) = \varepsilon_n |g(\lambda_n)|^{\frac{r}{2}}g(\lambda_n)^{-1}, \quad \| f \|^p_p \leq C \sum_{n=1}^N|g(\lambda_n)|^r \| S_{\lambda_n} \|_2^{-2}, \]

where $C$ does not depend on $N$. Combining the above relations we have that 
\begin{align*}
 \sum_{n=1}^N |g(\lambda_n)|^r \| S_{\lambda_n} \|_2^{-2} & \leq \left\| \sum_{n=1}^N f(\lambda_n) g(\lambda_n) \| S_{\lambda_n} \|_2^{-1} h_n \right\|_2^2 \\
 & = \| P_N(fg) \|_2^2 \\
 & \leq \| f g \|_2^2 \\
 & \leq \| f \|_p^2 \|g\|_r^2 \\
 & \leq C \Big( \sum_{n=1}^N|g(\lambda_n)|^r \| S_{\lambda_n} \|_2^{-2} \Big)^{\frac{2}{p}} \|g \|^2_r.
\end{align*}

Rearranging the above inequality the claim follows, since $C$ is independent of $N$.

\section{Carleson measures and the separately harmonic Hardy space in the bidisc}
\label{sec:Carleson}
In the introduction we alluded to the Chang - Stein characterization of Carleson measures for the separately harmonic Hardy spaces $h^p(\mathbb{D}^d)$. 
We shall take a closer look now. A rectangle $R\subseteq \mathbb{T}^d$ for us is going to be just a set of the form $I_1 \times \dots \times I_d$ where $I_m \subseteq \mathbb{T}$ are arcs. 
The corresponding product of regions $S(I_1)\times \dots \times S(I_d)$ will be denoted by $S(R)$ and we will call it a box. 
More generally, if $U\subseteq \mathbb{T}^d$ is an open set then 
\[ S(U):=\bigcup_{\text{rectangles} \,\, R \subseteq U} S(R). \]

As noted by Chang \cite{Chang1979} the following theorem can be extracted from \cite[p. 236]{Stein1970}. 

\begin{thm} 	Let $  1<p < \infty$ and $\mu$ a positive finite Borel measure on $\mathbb{D}^d$. 
	Then $h^p(\mathbb{D}^d) \subseteq L^p(\mathbb{D}^d,\mu)$ if and only if there exists a constant $C>0$ such that for every open set $U \subseteq \mathbb{T}^d$ 
	\[ \mu(U) \leq C |U|. \]

\end{thm}

To the best of our knowledge, it is not known if this characterization remains true for the holomorphic Hardy spaces. The sufficiency of the Chang-Stein condition is clear by inclusions.

In this part we will settle for a more modest objective. Consider the map $\sigma : \mathbb{D}^2 \to \mathbb{D}^2, \,\, \sigma(z_1,z_1)=(\overline{z_1},z_2).$ Given $\mu \in M_+(\mathbb{D}^2)$ we will denote by $\sigma_*\mu$ the pushforward of $\mu$ via the map $\sigma. $ That is, for every $\varphi \in L^1(\mathbb{D}^2,\mu)$ 
\[ \int_{\mathbb{D}^2}\varphi d\sigma_* \mu = \int_{\mathbb{D}^2}\varphi\circ\sigma d\mu. \]

\begin{lem}\label{lem:symmetric_carleson}
	Suppose that $ \mu \in M_+(\mathbb{D}^2)$ and $1<p<\infty.$ If $\mu$ and $\sigma_*\mu$ are Carleson for $H^p(\mathbb{D}^2)$ then $\mu$ satisfies the Chang-Stein condition. 
\end{lem}

\begin{proof}
Consider the  densely defined operator in $h^p(\mathbb{D}^2)$ given by the formula 
\begin{align*} \mathbb{P}\Big( \sum_{(n,m) \in \mathbb{Z}^2} a_{nm} r_1^nr_2^m e^{i(\theta_1n+ \theta_2 m)} \Big) & = \\ \frac14 a_{00} +  
\frac12 \sum_{n\in \mathbb{N}} a_{n0} r_1^n e^{i\theta_1 n} + \frac 12 \sum_{m\in \mathbb{N}} a_{0m}r_2^m e^{i\theta_2 m} & + \sum_{(n,m) \in \mathbb{N}^2} a_{nm} r_1^nr_2^m e^{i(\theta_1n+ \theta_2 m)}, \end{align*}
when $a_{nm}$ are all but finitely many equal to zero. 
	Due to the M. Riesz theorem this operator extends to a bounded operator from $h^p(\mathbb{D}^2)$ 
	onto $H^p(\mathbb{D}^2)$.
	Furthermore, for any $u \in h^p(\mathbb{D}^2)$ we have that 
	\[ u = \mathbb{P}u + \overline{\mathbb{P}\overline{u} } + (\mathbb{P}(u\circ \sigma) )\circ \sigma + \overline{  (\mathbb{P}(\overline{u}\circ \sigma)) \circ \sigma}.    \]
        In order to finish the  proof it remains to notice the obvious fact that $\sigma$ preserves the Lebesgue measure. Therefore for the third term for example we have 
	\begin{align*}
		\int_{\mathbb{D}^2}|(\mathbb{P}u\circ \sigma )\circ \sigma|^p d\mu & = \int_{\mathbb{D}^2} | \mathbb{P}u\circ \sigma |^p d\sigma_* \mu \\
		\leq & C \Vert \mathbb{P}( u \circ \sigma) \Vert^p_p \\
		\leq & C \Vert u \circ \sigma \Vert^p_p \\
		= & C \Vert u \Vert^p_p.
	\end{align*}
Here and in the sequel the constant $C$ might change from appearance to appearance. 

	Similarly all four terms are bounded by $\Vert u\Vert_p^p$ times a multiplicative   constant that depends only on $p$ and $\mu$. 
	Therefore $\mu $ is a Carleson measure for the separately harmonic Hardy space $h^p(\mathbb{D}^2)$ and hence it satisfies the Chang-Stein condition. 
\end{proof}
\begin{rem}
    A straightforward adaptation of the argument above yields the same result for all $d>2$. Namely, if $\mu$ is a measure on $\D^d$ such that $\mu, {\sigma_1}_*\mu,\dots, {\sigma_{d-1}}_*\mu$ are Carleson for $H^p(\D^d)$, then $\mu$ is Carleson for $h^p(\D^d)$. Here $\sigma_i$ denotes the map the conjugates the $i$-th variable. 
\end{rem}
Next we shall need another elementary observation. 

\begin{lem}
\label{lemma:interpolation_carleson}
	Let $1 < p < q < \infty$, then any Carleson measure for $H^p(\mathbb{D}^d)$ is a Carleson measure for $H^q(\mathbb{D}^d).$   
\end{lem}
\begin{proof}
	Notice that if $0<\theta<1$ and we denote by $(H^p(\mathbb{D}^d),H^q(\mathbb{D}^d))_\theta$ the  $\theta$ interpolating space constructed with the complex method, it coincides, with equivalence of norms with the Hardy space $ H^s(\mathbb{D}^d) $ where $(1-\theta) p^{-1}+ \theta q^{-1}=s^{-1}$. That is because the compatible pair of Hardy spaces can be identified isometrically with a compatible pair of subspaces of $L^1(\mathbb{T}^d,m)$, and because the Riesz projection on the first quadrant is a bounded operator on all $L^p(\mathbb{T}^d,m)$ for $p>1$.

	To finish the proof if suffices to notice that if under the hypothesis of the theorem, and for any $n\in \mathbb{N}$ we have 
	\[ \int_{\mathbb{D}^d}|f|^{np}d\mu = \int_{\mathbb{D}^d}|f^n|^p d\mu \leq C_\mu \Vert f\Vert^{np}_{np}. \]
	Consequently, the identity operator $\id : H^{np}(\mathbb{D}^d)\to L^{np}(\mathbb{D}^d,\mu) $ is bounded for every $n\in \mathbb{N}$ and therefore by interpolation it is bounded for ever $q>p.$   
\end{proof}

 \begin{coro}
\label{coro:carlesonmeasures}
    A sequence $\Lambda =(\lambda_n)_n$ in the bidisc is Carleson for $h^p(\mathbb{D}^2)$ if and only if both $Z$ and $\sigma(\Lambda):=(\sigma(\lambda_n))_n$ are Carleson sequences for $H^p(\mathbb{D}^2)$. In particular, if $\sigma(\Lambda)= \Lambda$ and $q, p>1$, then $\Lambda$ is a Carleson measure for $H^p(\mathbb{D}^2)$ if and only if it is a Carleson measure for $H^q(\mathbb{D}^2)$.

\end{coro}
\begin{rem}[Random Carleson Sequences]
   The Kolmogorov  $0-1$ law for random Carleson sequences for $H^2(\mathbb{D}^d)$ was described in \cite{Chalmoukis2024}, where a random sequence in the polydisc is chosen by picking deterministic radii and independent random arguments distributed uniformly on $\mathbb{T}^d$. Since the distribution law of such random arguments is unchanged under the transformation $\sigma$,  Carleson sequences for the separately harmonic and holomorphic Hardy spaces coincide for all $d$ almost surely.
\end{rem}

\section{Modified Szeg\"o kernels in the bidisc}
\label{sec:FangXia}
In this section we will develop the idea which is necessary in order to show that $\Lambda_0$ in Theorem \ref{thm:main} is simply interpolating for all $1\le p\le2$.

  For our purposes it is convenient to consider a simple modification of the Szegö kernel $S$ in \eqref{eqn:szego}. 
Given a parameter $t>0$ and $w\in \mathbb{D}^2$ we define 
\[ \psi_{w,t}:=\frac{S_w^{1+t}}{\Vert S_w\Vert_2^{1+2t}}. \]
Almost orthogonality properties for collections of such modified Szegö kernels have been studied in \cite{FangXia15} in the setting of the Drury-Arveson space on the unit ball. In particular, as a corollary of \cite[Lemma 5.1]{FangXia15} one deduces the following result for the Hardy space on the polydisc, which can also be obtained by direct computation.
\begin{lem} \label{prop:psi_estimate}
	For all $t>0$ there exists a constant $C=C(d, t)>0$ such that 
	\[ | \langle \psi_{z,t}, \psi_{w,t} \rangle| \leq C| \langle g_z , g_w \rangle |^{1+t} \]
\end{lem}

The idea that will guide us from this point on is that the sequence $\Lambda_0$ in Theorem \ref{thm:main} is going to be constructed as a union of finite collections of points which are sufficiently disjoint from each other so that the properties of one collection do not essentially affect the properties of the other.
Furthermore, each such finite collection of points will have uniformly controlled simply interpolating constant while it is going to have increasingly large Carleson constant.

In order to do so, we will show that if our sequence is of a very particular form, then we can use the modified kernels $\psi_{z,t}$ in order to construct an approximate right inverse of the restriction operator $T^p_\Lambda$. If the approximation is good enough a Neumann series argument allows us to find an exact right inverse of the restriction operator with good control of its norm.

Our next proposition will contain the basic idea of our construction. 
It will be necessary to introduce some more notation before discussing it. 
If $\Lambda=(\lambda_n)$ is a sequence of points in the bidisc we will write $g_n$ instead of $g_{\lambda_n}$ for brevity, where $g_{\lambda_n}$ is the normalized Szegö kernel from \eqref{eqn:normalized_Szego}. Furthermore, the matrix $G = \langle g_n, g_m \rangle$ will be called the {\it Gram matrix} of the sequence or the {\it Gramian} of $\Lambda.$

It can be shown \cite[Chapter 9]{AM02} that a sequence is weakly separated if and only if $\sup_{n\neq m} |\langle g_n, g_m  \rangle| < 1$. 
We shall call the quantity
\[
	\gamma = \sqrt{1-\sup_{n\neq m}|\langle g_n, g_m \rangle|^2} > 0
\]
 the constant of weak separation. 
A related notion,  is that of {\it column boundedness} of a sequence. 
$\Lambda $ is called column bounded if the columns of the corresponding Gramian form a bounded set in $\ell^2(\mathbb{N}).$ 
As before the constant   
\[ \Delta := \sup_m \sum_{n:n\neq m}^{\infty} |\langle g_n, g_m \rangle|^2 < + \infty, \]
will be called the column boundedness constant of $\Lambda.$ 
We will say that a sequence is $\gamma$ weakly separated if the weak separation constant of the sequence is less or equal to $\gamma $. The phrase $\Delta$ column bounded has a similar meaning.

This leads us to the following proposition.

\begin{prop}
\label{prop:bound_inv}
	For all $\gamma, \Delta>0$ there exists a constant $C=C(\gamma,\Delta)>0$ such that, if $1\le p \le2$ and $ \Lambda \subseteq \mathbb{D}^2$ is a finite collection of points which is  $\gamma$  weakly separated and $\Delta$ column bounded  and also satisfies 
	\[ \Vert S_\lambda \Vert_2 = \text{constant on} \,\,\, \Lambda, \]
	then there exists a bounded linear operator $R: \ell^p \to H^p(\mathbb{D}^2)	 $ which is a right inverse of $T_\Lambda^p: H^p(\mathbb{D}^2) \to \ell^p $ and $\Vert R\Vert_{\ell^p\to H^p(\mathbb{D}^2)}\leq C. $   
\end{prop}

\begin{proof}
 Let $t>0$ to be determined later and notice that 
 \begin{equation}\label{eq:power_column_bound}
		\sum_{n:n\neq m} | \langle g_n, g_m \rangle|^{1+t} \leq (1-\gamma^2)^{\frac{t-1}{2}}\sum_{n:n\neq m} | \langle g_n, g_m \rangle|^2 \leq \Delta (1-\gamma^2)^{\frac{t-1}{2}}.  
	\end{equation}
	Since no confusion arises, we will continue to denote by $\ell^p$ the space $\mathbb{C}^N$, where $N$ is the number of points of the sequence $\Lambda$, equipped with the $\ell^p$ norm. 
	Then we consider the operator $B_t:\ell^p \to H^p(\mathbb{D}^2)	$ defined as 
	\[ B_t(a) = \sum_{n=1}^{N} a_n \psi_{n,t}, \,\, a=(a_1,\dots,a_N)\in \ell^p,  \]
where $ \psi_{n,t}= \psi_{\lambda_n,t}.$ 

Let us start by estimating $\Vert B_t\Vert_{\ell^2 \to H^2(\mathbb{D}^2)} $. We apply first Lemma \ref{prop:psi_estimate} and subsequently Cauchy-Schwarz and the estimate \eqref{eq:power_column_bound} in order to get that, for some $C$ depending on $t$, we have 
\begin{align*} \Vert B_t(a)\Vert^2_{2}  & = \sum_{n,m=1}^N a_n \overline{a_m} \langle \psi_{n,t}, \psi_{m,t} \rangle \\
	& \leq C \sum_{n,m=1}^N |a_n a_m| |\langle g_n, g_m \rangle|^{1+t} \\
	& \leq C \sup_{m} \Big( \sum_{n=1}^{\infty} | \langle g_n, g_m \rangle|^{1+t} \Big)  \Vert a\Vert_{\ell^2}^2 \\ 
	& \leq C(1+\Delta (1-\gamma^2)^{\frac{t-1}{2}}) \Vert a\Vert_{\ell^2}^2.
 \end{align*}

 In order to estimate  $\| B_t\|_{\ell^1 \to H^1(\mathbb{D}^2)}$ set $s:=p/2(1+t)-1$ and notice that, $s>0$ for all $t>1$ and that

\begin{equation}
\begin{split}
\label{eq:pnorm_psit}\Vert \psi_{z,t} \Vert_p^p  &=\|\psi_{z, t}^\frac{p}{2}\|^2_2\\
&=\left\|\frac{S_z^{1+s}}{\|S_z\|_2^{\frac{p}{2}(1+2t)}}\right\|_2^2\\
&=\|\psi_{z, s}\|^2_2\cdot\|S_z\|^{-p(1+2t)+2+4s}\\
&\leq C\|S_z\|_2^{p-2}
\end{split}
\end{equation}

thanks to Lemma \ref{prop:psi_estimate}, $z=w$. In particular for $a\in \ell^1$
\[ \Vert B_t(a)\Vert_{\ell^1} \leq \Vert a\Vert_{\ell^1} \sup_{1\leq n \leq N} \Vert \psi_{n,t}\Vert_{1}  \leq C \|S_{\lambda_1}\|^{-1}_2 \Vert a\Vert_{\ell^1},  \]
since $\|S_{\lambda_n}\|_2$ is constant on $\Lambda$. Notice that by considering $H^p(\mathbb{D}^2)$ as a subspace of $L^p(\mathbb{T}^2,m)$ we can apply the Riesz-Thorin interpolation theorem to conclude that $\Vert B_t\Vert_{\ell^p \to H^p(\mathbb{D}^2)} \leq C\|S_{\lambda_1}\|_2^{1-\frac{2}{p}}$, where $C$ depends only on $\Delta, \gamma$ and $t.$   

Next we consider the operator $T_\Lambda^p B_t : \ell^p \to \ell^p$, which for $a\in \ell^p$ acts as follows 
\begin{align*}
	T_\Lambda^p B_t (a) & = T_\Lambda^p \Big ( \sum_{n=1}^N a_n \psi_{n,t} \Big) \\
			    & = \Big( \sum_{n=1}^N a_n \psi_{n,t}(\lambda_m) \Vert S_{\lambda_m}\Vert_2^{-\frac2p} \Big)_{m=1}^N  \\
			    & =  \Big( \sum_{n=1}^N a_n \frac{S_{\lambda_n}(\lambda_m)^{1+t}}{\Vert S_{\lambda_n}\Vert_2^{1+2t+\frac2p}}   \Big)_{m=1}^N \\
			    & =  \Big( \Vert S_{\lambda_1}\Vert_2^{1-\frac2p} \sum_{n=1}^N a_n \langle g_n, g_m \rangle^{1+t}  \Big)_{m=1}^N,
\end{align*}
where we have used the fact that the norm of the Szeg\"o kernel vectors is constant on the sequence.  

Therefore we conclude that 
\[
\Vert \Vert S_{\lambda_1}\Vert_2^{\frac2p-1} T_\Lambda^p B_t - Id \Vert_{\ell^p\to \ell^p } \leq \|(\langle g_n, g_m\rangle^{1+t})_{n, m=1}^N\|_{\ell^p\to\ell^p}\leq \Delta(1-\gamma^2)^{\frac{t-1}{2}} 
\]
via an applicarion of \eqref{eq:power_column_bound} and Riesz's interpolation theorem.
Choosing $t>0$ such that the latter expression equals $ \frac12 $ we conclude that the operator $T_\Lambda^pB_t $ is invertible on $\ell^p $. Moreover, the norm of its inverse is controlled by $\|S_{\lambda_1}\|^{\frac{2}{p}-1}_{2}$, up to a constant that is independent of $p$. Hence $R : = B_t(T_\Lambda^p B_t)^{-1} : \ell^p \to H^p(\mathbb{D}^2) $ is a right inverse of $T_\Lambda^p  $ and
\[ \Vert R \Vert_{\ell^p \to H^p(\mathbb{D}^2)} \leq \Vert B_t\Vert_{\ell^p \to H^p(\mathbb{D}^2)} \Vert (T_\Lambda B_t)^{-1}\Vert_{\ell^p\to \ell^p}\leq C, \]
where $C$ does not depend on $p$. This concludes the proof.
\end{proof}
\section{Carleson's quilt construction revisited}\label{sec:Carlesons_Quilt}

In this section we are going to provide the second main tool for the proof of Theorem \ref{thm:main}, by using Carleson's example of a measure on the bidisc which satisfies the one box condition 
but it is not Carleson for $h^p(\mathbb{D}^2) $ for any $p>1 $. 

First we shall describe briefly Carleson's construction. For convenience we will identify $\mathbb{T}^2 $  with $(\mathbb{R}/\mathbb{Z})^2 $. In these coordinates the map $\sigma $ is give by $\sigma(\theta_1, \theta_2) = (1-\theta_1,\theta_2). $  We will generally follow the notation and terminology as in \cite{Tao} but also one can consult the original exposition \cite{Carleson74}.

For us  an interval $ I \subseteq (\mathbb{R}/\mathbb{Z})^2 $ is dyadic if its end points are consecutive points of the set 
$\mathcal{D}: = \{ \frac{j}{2^N}: j\in \mathbb{Z}, N \in \mathbb{N}\cup \{0\} \}$. 
Let $\mathcal{R} $ be a finite collection of dyadic 
 rectangles in the unit square $ (0,1)^2 $ such that 
 \begin{equation}\sum_{R \in \mathcal{R}} |R| = 1  \end{equation}
and also 
\begin{equation} \label{eq:quilt_property}
	\sum_{R\in \mathcal{R}: R\subseteq Q}|R| \leq |Q|
\end{equation}
for every dyadic  rectangle $ Q \subseteq \mathbb{T}^2. $
We will call such a collection a {\it quilt} \cite[Section 7]{Tao}.
The total area of the quilt is simply $\Big| \bigcup_{R\in \mathcal{R}} R \Big| $. We will say that a quilt is $\sigma $ invariant
if $\mathcal{R} = \{ \sigma(R) : R \in \mathcal{R}\} $ and that it is equiareal if every $ R\in \mathcal{R} $ has the same 
area.

The singleton $ \{ [0,1)^2 \} $ is trivially a quilt, but it is not obvious that quilts with arbitrarily small area exist. This is the content of the following theorem. 

\begin{thm}[Carleson \cite{Carleson74}]
	Given $\varepsilon > 0 $  there exists an equiareal quilt of total area smaller than $\varepsilon. $ 
\end{thm}
 
Although the fact that the rectangles in Carleson's construction have the same area is not explicitly stated, it is immediate after an inspection of the proof, since the rectangles are obtained by applying iteratively linear maps of constant determinant to the square $[0,1)^2. $ 

Finally, to the quilt $\mathcal{R} $ we associate a finite collection of points in the bidisc in the following way. 
To any $R = I \times J \in \mathcal{R} $, we associate the point 
\begin{equation}
\label{eqn:z_R}
z_R := (\sqrt{1-|I|}e^{i\theta_1},\sqrt{1-|J|}e^{i\theta_2}), 
\end{equation}
where $\theta_1,\theta_2 $ are
the midpoints of $I,J $ respectively. Set $\Lambda(\mathcal{R}) = \{ z_R : R\in \mathcal{R} \} $. 
One can verify that if $\mathcal{R} $ is the collection of all dyadic rectangles then $\Lambda(\mathcal{R}) $ 
is a weakly separated sequence. Therefore any sequence associated to a quilt is 
in particular weakly separated, provided that all the rectangles in the quilt are different.




The next proposition is going to connect the quilt property \eqref{eq:quilt_property} 
with the column boundedness property. 

\begin{lem}\label{lem:colum_box}
 There exists an absolute constant $C>0 $ such that for every positive Borel
 measure $\mu $ on $\mathbb{D}^2 $ we have 

 \[ \sup_{z\in \mathbb{D}^2} \Vert g_z \Vert_{L^2(\mathbb{D}^2,\mu)} \leq C 
 \sup \frac{\mu(S(R))}{|R|}, \]
 where the second supremum is taken over all dyadic rectangles.

\end{lem}

\begin{proof}
Let us call $M = \sup\{ \mu(S(R))|R|^{-1}: R \,\, \text{dyadic rectangle} \}$.   
First notice that any rectangle $I\times J $  can be covered by $4=2^2$ disjoint dyadic rectangles $I_i\times J_i $ so that $ |I| \leq |I_i| \leq 2|I| $ and $ |J| \leq |J_i| \leq 2|J|. $  
	Therefore for a general rectangle $ Q \subseteq \mathbb{T}^2 $ 
 \[ \frac{\mu(S(Q))}{|Q|} \leq 2^2\cdot 2^2 M.   \]
Next, fix $ z\in \mathbb{D}^2 $ and consider the rectangle $Q=I\times J$ so that $z_Q$ as in \eqref{eqn:z_R} concides with $z $.
Define, for any multi-index $j = (j_1,j_2) \in \mathbb{N}^2 $ the dilated rectangle
$2^jQ $ as the rectangle with the same center point as $R $, and having sides 
of lengths $2^{j_1}|I|$ and  $2^{j_2}|J|$. 
Set for all $m\in\mathbb{N} $, 
\[ A_m:=\bigcup_{|j|\leq m}S(2^j Q). \]
For $m $ large enough, $A_m =\mathbb{D}^2 $. If $ A^\circ_m : = A_m \setminus A_{m-1} $, it is then an elementary calculation  to prove that 
\[ |S(z,w)| \leq C \frac{1}{2^m|Q|}, \,\,\, \forall w\in A_m^\circ. \]

Also 
\[ \mu(A_m^\circ) \leq \mu(A_m) \leq M|Q| \sum_{|j|\leq m}2^{|j|} \leq C M | Q | m2^m. \]

 Thus 
\begin{align*}
	\Vert S_z\Vert^2_{L^2(\mathbb{D}^2,\mu)} \leq C \frac{1}{|Q|^2} \sum_{m\in \mathbb{N}} \frac{\mu(A_m^\circ)}{2^{2m} } \leq C M|Q|^{-1} \leq C M \Vert S_z\Vert^2_2.
\end{align*}
\end{proof}

Summarizing the above lemmata we have proved the following proposition 
\begin{coro}
\label{coro_finite_pieces}
	There exist a $0<\gamma<1$ and a $\Delta>0$ such that, given any $\varepsilon >0$,
there exists a finite 
$\sigma $ invariant	sequence of points $\Lambda \subset \mathbb{D}^2 $ such that 
	\begin{itemize}
		\item[(a)] $\Lambda $ is $\gamma $ weakly separated
		\item[(b)] $\Lambda $ is $\Delta $ column bounded
			
		\item[(c)] We have that 
			\[ \sup_{V \subseteq \mathbb{T}^2, \, \text{open}} \frac{\mu_\Lambda(S(V))}{|V|} \geq \varepsilon^{-1}. \]
			
		\item[(d)] The norm of the Szeg\"o kernel $ \Vert S_\lambda\Vert_2 $ is constant on $\Lambda. $ 
	\end{itemize}
\end{coro}

\begin{proof}
	Consider an equiareal quilt $\mathcal{R} $ of total area less or equal to $ \varepsilon $ and let $ \Lambda : = \Lambda(\sigma(\mathcal{R})\cup \mathcal{R}). $ By construction, the  sequence $\Lambda  $ is $\sigma $ invariant.  Recall that any  sequence $\Lambda $ constructed in this manner is weakly separated by a uniform constant $\gamma $, hence $(a) $ 
	follows.
	Let $\mu_\Lambda $ the measure associated to the sequence and notice that by Lemma \ref{lem:colum_box} and the quilt property \eqref{eq:quilt_property} we have 
	\[ \sum_{n=1}^\infty|\langle g_n, g_m \rangle|^2 = \Vert g_m\Vert^2_{L^2(\mathbb{D}^2,\mu_\Lambda)} \leq \Delta + 1, \]
	where $\Delta $ depends only on the one-box constant of the measure $\mu_\Lambda$.
	In order to prove $(c) $ let $V = \bigcup_{R\in \mathcal{R}}R $. Then 
	\[ \frac{\mu_\Lambda(S(V))}{|V|} \geq \varepsilon^{-1} \sum_{\lambda\in\Lambda} \Vert S_\lambda \Vert^2_2 = \varepsilon^{-1}\sum_{R\in \mathcal{R}}|R| = \varepsilon^{-1}. \]
  Finally, property $(d) $ comes from the fact that the quilt is equiareal.  	
\end{proof}

\section{Proof of Theorem \ref{thm:main}}
\label{sec:proof}

Clearly, it is enough to argue in the case $d=2$. Let, for all $L\in\R_+$, 
\[
\varphi_L(\xi):=(1/2, 1/2)+\frac{\xi-(1/2, 1/2)}{L}\qquad \xi\in\T^2
\]
be the homothety that rescales $\T^2$ onto a square centered at $(1/2, 1/2)$ of side parallel to the ones of $\T^2$ of length $1/L$. Let $\Lambda_M$ be the finite collection of points from Corollary \ref{coro_finite_pieces}, $\varepsilon_M=1/M$, and let $\quilt_M$ be its associated rectangles.  Consider the collection of rectangles
\[
\Phi:=\bigcup_{M\in\mathbb{N}}\varphi_{L_M}(\quilt_M),
\]
where  $(L_M)_M$ is a positive sequence to be chosen later. Set $\Lambda_0:=\Lambda(\Phi)$. Then by Corollary \ref{coro_finite_pieces} (c), $\Lambda_0$ does not satisfy the Chang-Stein condition. Moreover, by construction $\Lambda$ is $\sigma$-invariant, hence it is not a Carleson sequence for all $H^p(\mathbb{D}^2)$, via an application of Lemma \ref{lem:symmetric_carleson}. We are then left to show that $\Lambda_0$ is simply interpolating for $H^p(\mathbb{D}^2)$ for all $1\leq p\le2$, provided that $(L_M)_M$ diverges fast enough. First note that the simply interpolation constant of each finte collection $\Lambda_M:=\Lambda(\varphi_{L_M}(\quilt_M))$ is bounded by a constant $C$ uniformly, thanks to Proposition \ref{prop:bound_inv}.

Let, for all $i\in\N$, $Z_i\subset\D$ be the collection of the projection on the first variables of all points in $\Lambda_i$, and set 
\[
\Theta_i(z):=\prod_{\lambda\in Z_i}\frac{\lambda-z}{1-\overline{\lambda}z}\qquad z\in\D.
\]
Note that the number of points in each $\quilt_M$ is independent of the sequence $(L_M)_M$. Hence if $(L_M)_M$ diverges fast enough, one has
\[
\inf_{z\in\D, i\in\N}\left\{|\Theta_i(z)|+\prod_{j\ne i}|\Theta_j(z)|\right\}\ge\delta>0.
\]

Hence, \cite[Theorem 3.2.14]{Nikolski02}, for all bounded sequences $(w_n)_n$ there exists a bounded analytic function $\varphi$ on the unit disc such that 
\[
\varphi_{|Z_i}=w_i\qquad i\in\N,
\]
and once we set $\phi(z_1, z_2):=\varphi(z_1)$ we obtain
\[
\phi_{|\Lambda_i}=w_i\qquad i\in\N.
\]
The argument in \cite[Theorem 2.2, p. 288]{Garnett} ( see \cite[Theorem 4.1]{Dayan20} for an adaptation to similar generalized interpolation problems) yields the existence, for all $M$, of functions $\phi_1, \dots, \phi_M$ on $\D^2$ such that ${\phi_i}_{|\Lambda_j}=\delta_{i, j}$ and 
\begin{equation}
\label{eqn:rowmult}
\sup_{z\in\D^2}\sum_{i=1}^M|\phi_i(z)|\leq B,
\end{equation}
 where $B=B_\delta$ is independent of $M$. Let $C_M$ denote the simply interpolation constant of the collection $\bigcup_{i=1}^M\Lambda_i$. We are left to show that 
\begin{equation}
\label{eqn:sup:C_M}
\sup_M C_M<\infty
\end{equation}
To this end, let $(a_n)_{n=1}^{|\bigcup_{i=1}^M\Lambda_i|}$ be a collection of targets, and set $A_i:=\sum_{\lambda_n\in\Lambda_i}|a_n|^p$. Thanks to Proposition \ref{prop:bound_inv}, for all $i=1, \dots, M$ there exists a function $f_i$ on the bidisc so that $\|f_i\|_{H^p(\D^2)}\le C A_i^\frac{1}{p}$, and  $f_i(\lambda_n)=a_n(1-|\lambda_n|^2)^{-\frac{1}{p}}$ for all $\lambda_n\in\Lambda_i$. Hence the function 
\[
f=\sum_{i=1}^M \phi_if_i
\]
interpolates the values $a_n(1-|\lambda_n|^2)$ at the points of $\bigcup_{i=1}^M\Lambda_i$, and thanks to \eqref{eqn:rowmult} one has that 
\[
\|f\|_{H^p(\D^2)}\leq B\left(\sum_{i=1}^M\|f_i\|_{H^p(\D^2)}^p\right)^\frac{1}{p}\leq B C \|a\|_{\ell^p}, 
\]
showing \eqref{eqn:sup:C_M}.
This concludes the proof of Theorem \ref{thm:main}.

\bibliography{biblio.bib}

\bibliographystyle{abbrv}

\end{document}